\documentclass[12pt]{amsart}
\usepackage{txfonts}
\theoremstyle{plain}
\newtheorem{theo}{Theorem}
\newtheorem{Theorem}[theo]{Theorem}
\newtheorem{lemm}[theo]{Lemma}
\newtheorem{Lemma}[theo]{Lemma}

\newtheorem{Corollary}[theo]{Corollary}

\theoremstyle{remark}
\newtheorem{Remark}[theo]{Remark}
\newtheorem*{Remark*}{Remark}
\newtheorem{Example}[theo]{Example}

\theoremstyle{definition}

\def\C{\mathbb C}
\def\Z{\mathbb Z}
\def\R{\mathbb R}
\def\D{\mathbb D}
\def \CP {{\mathbb C}{\mathbb P}}
\def\e{\mathbf e}

\DeclareMathOperator{\Lie}{\mathrm{Lie}}

\DeclareMathOperator{\Hom}{\mathrm{Hom}}
\DeclareMathOperator{\pos}{\mathrm{pos}}

\def \ssminus {\smallsetminus}
\def \calO {{\mathcal O}}

\begin{document}

\title[Torus actions on complex manifolds]
{Completely integrable torus actions on complex manifolds with fixed points}

\author[H.~Ishida]{Hiroaki Ishida}
\address{Osaka City University Advanced Mathematical Institute, 3-3-138, Sumiyoshi-ku, 
Osaka 558-8585, Japan.}
\email{ishida@sci.osaka-cu.ac.jp}

\author[Y.~Karshon]{Yael Karshon}
\address{Dept.\ of Mathematics, University of Toronto, 
40 St.George Street, Toronto Ontario M5S 2E4, Canada}
\email{karshon@math.toronto.edu}

\date{\today}
\thanks{The first author is supported by JSPS Research Fellowships for Young 
Scientists. This work is partially supported by the JSPS Institutional Program for Young Researcher Overseas Visits ``Promoting international young researchers in mathematics and
mathematical sciences led by OCAMI"}
\thanks{The second author is partially supported by 
the Natural Sciences and Engineering Research Council of Canada}

\begin{abstract}
We show that if a holomorphic $n$ dimensional compact torus action
on a compact connected complex manifold of complex dimension $n$ 
has a fixed point
then the manifold is equivariantly biholomorphic to a smooth 
toric variety.
\end{abstract}

\maketitle

\section{Introduction}
\label{sec:intro}

We begin by recalling some notions from the theory of toric varieties. 

We work in the vector space $\Lie(S^1)^n \cong \R^n$
with the lattice $\Hom(S^1,(S^1)^n) \cong \Z^n$.
Here, we identify $\Lie(S^1)$ with $\R$ such that 
the exponential map $\exp \colon \R \to S^1$
is $t \mapsto e^{2 \pi i t}$.

A \emph{unimodular fan} is a finite set $\Delta$ of convex polyhedral cones
with the following properties.
\begin{enumerate}
\item  A face of a cone in $\Delta$ is also a cone in $\Delta$.
\item  The intersection of two cones in $\Delta$ is a common face.
\item  Every cone in $\Delta$ is unimodular, i.e., 
it has the form $\pos(\lambda_1,\ldots,\lambda_k)$
where $\lambda_1,\ldots,\lambda_k$ is part of a $\Z$-basis of the lattice.
Here, $\pos$ denotes the positive span: the set of linear combinations
with non-negative coefficients.\footnote{
This property of a cone or a fan is also described in the literature
by the adjectives
\emph{smooth}, \emph{non-singular}, \emph{regular}, and \emph{Delzant}.}
\end{enumerate}
A fan $\Delta$ is \emph{complete} if the union of the cones in $\Delta$
is all of $\Lie(S^1)^n$.

The theory of toric varieties associates to a unimodular fan $\Delta$
a complex manifold $M_\Delta$
with a holomorphic $(\C^*)^n$-action with the following properties.
\begin{enumerate}
\item 
The fixed points in $M_\Delta$ are in bijection with the $n$-dimensional
cones in $\Delta$. 
\item
Let $p$ be a fixed point in $M_\Delta$.
Then the isotropy weights at $p$ are a $\Z$-basis to 
the lattice $\Hom((S^1)^n,S^1) \subset (\Lie(S^1)^n)^*$.
Moreover, let $\lambda_1,\ldots,\lambda_n$ be the dual basis;
then the cone in $\Delta$ that corresponds to $p$
is $\pos(\lambda_1,\ldots,\lambda_n)$.
\item 
The manifold $M_\Delta$ is compact if and only if the fan $\Delta$ is complete.
\end{enumerate}
Explicitly, let $\lambda_1,\ldots,\lambda_m \in \Z^n$
be the primitive generators of the one dimensional cones in $\Delta$.
Each $\lambda_i$ encodes a homomorphism $a \mapsto a^{\lambda_i}$
from $\C^*$ to $(\C^*)^n$; together they give a homomorphism
$\pi \colon (a_1,\ldots,a_m) \mapsto \prod_{j=1}^m a_j^{\lambda_j}$
from $(\C^*)^m$ to $(\C^*)^n$.
Then $M_\Delta = U_\Delta/K_\Delta$, where
$ U_\Delta = \{ z \in \C^m \ | \ \pos(\lambda_i \, | \, z_i = 0) \in \Delta 
   \} $
and $ K_\Delta = \ker \pi $.
For the details of the construction and the proof of its properties,
we refer the reader to the book~\cite{cox} by Cox, Little, and Schenck
and to the book~\cite{audin} by Audin.

In fact, $M_\Delta$ is an \emph{algebraic} variety.  Moreover,
every smooth complex algebraic variety that is equipped with 
an algebraic $(\C^*)^n$-action
with an open dense free orbit is isomorphic to some $M_\Delta$.
(The proof of this fact appeared in the book \cite{KKMS} 
by Kempf, Knudsen, Mumford, and Saint-Donat 
and in the article \cite{miyake-oda} by Miyake and Oda
and relies on a lemma of Sumihiro \cite{sumihiro};
see Corollary 3.1.8 in~\cite{cox}.)
Our main theorem is a complex analytic variant of this result:

\begin{Theorem} \label{main-theorem}
Let $M$ be a connected complex manifold of complex dimension $n$,
equipped with a faithful action of the torus $(S^1)^n$ by biholomorphisms.
If $M$ is compact and the action has fixed points, 
then there exists a unimodular fan $\Delta$
and an $(S^1)^n$-equivariant biholomorphism of $M_\Delta$ with $M$.
\end{Theorem}

\begin{Remark}\  
\label{remarks}
\begin{enumerate}
\item
Our theorem gives a negative answer to a question that was raised 
by Buchstaber and Panov in \cite[Problem 5.23]{buchstaber-panov}.

Let $M$ be a closed $2n$ dimensional manifold 
with an $(S^1)^n$-action that is locally standard: 
every orbit has a neighbourhood that is equivariantly diffeomorphic,
up to an automorphism of $(S^1)^n$,
to an invariant open subset of $\C^n$ with the standard $(S^1)^n$-action. 
Also assume that the quotient $M/(S^1)^n$ is diffeomorphic, 
as a manifold with corners, to a simple convex polytope $P$ in $\R^n$.\footnote{
A map from $M/(S^1)^n$ to $P$ is a diffeomorphism of manifolds with corners
if and only if it is a homeomorphism and, for every real valued function
on $P$, the function extends to a smooth function on $\R^n$
if and only if its pullback to $M$ is smooth.
For every $k \in \{0 , \ldots,  n \} $,
a diffeomorphism carries the $k$ dimensional orbits in $M$
to the relative interiors of the $k$ dimensional faces of $P$.
}
Such manifolds, introduced in \cite{davis-januszkiewicz}
and studied in the toric topology community,
are called \emph{quasi-toric manifolds}\footnote{
Davis-Januszkiewicz~\cite{davis-januszkiewicz} used the term
\emph{toric manifold}, but this term was already used in the literature
to mean a smooth toric variety, so Buchstaber-Panov \cite{buchstaber-panov} 
introduced instead the term \emph{quasitoric manifold}. 
}.

The question of Buchstaber and Panov is 
whether there exists a non-toric quasitoric manifold 
that admits an $(S^1)^n$-invariant complex structure.

\item
Our theorem strengthens an earlier result of Ishida and Masuda,
that if a closed complex manifold of complex dimension $n$
admits an $(S^1)^n$-action,
and if its odd-degree cohomology groups vanish,
then the Todd genus of the manifold is equal to one.
See~\cite[Theorem 1.1 and Remark 1.2]{ishida-masuda}.

\item 
In Theorem~\ref{main-theorem}, the assumption ``complex"
cannot be weakened to ``almost complex".
For example, for every two complex toric manifolds of complex dimension~2,
their equivariant connected sum along a free orbit
supports an invariant almost complex structure, has fixed points,
but is not (equivariantly diffeomorphic to) a toric manifold;
see \cite[\S 11.2]{k:thesis}.
For higher dimensional analogues, see \cite[\S 13]{GK};
for more interesting four dimensional examples, 
see \cite[Theorem~5.1]{masuda}.
A necessary and sufficient condition for a quasitoric manifold
to admit an invariant almost complex structure
was given in~\cite[Theorem 1]{kustarev}.

\item 
The symplectic analogue of Theorem~\ref{main-theorem}
is also true: a closed symplectic manifold of dimension $2n$
with a faithful $(S^1)^n$ action with at least one fixed point
is a symplectic toric manifold.
To see this, it is enough to show that such an action is Hamiltonian;
being a toric manifold then follows 
from Delzant's theorem \cite[Th\'{e}or\`{e}me 2.1]{delzant}.
Let $p$ a fixed point.  There exist $n$ subcircles of $(S^1)^n$
that span $(S^1)^n$ and whose isotropy weights are all positive.
In order to show that the  $(S^1)^n$  action is Hamiltonian,
it is enough to show that each of these $S^1$ actions has a momentum map.
Fix one of these $S^1$ actions. Because there is a fixed point, the $S^1$ orbits
are null-homotopic, so the $S^1$ action lifts to an $S^1$ action on the
universal bundle, $\tilde{M}$.
Because $H^1(\tilde{M}) = 0$, this lifted action is Hamiltonian.
By Morse theory, at most one point of  $\tilde{M}$  can be a strict local
minimum for the momentum map
(see, e.g., \cite{guillemin-sternberg:convexity}).
So the fibre of  $\tilde{M}$  over the fixed
point  $p$  can contain only one point.
So  $\tilde{M} = M$, and so there is a momentum map on M.


\item
It is necessary to assume that the action has fixed points:
the complex torus $\C^* / (z \sim 2z)$ has a holomorphic $S^1$-action,
induced from multiplication on $\C^*$, but it is not a toric variety:
the $\C^*$-action is not faithful.

\item
It is necessary to assume that the manifold is compact:
the open unit disc in $\C$ with the natural circle action 
has a fixed point, but it is not a toric variety:
the circle action does not extend to a $\C^*$-action.
\end{enumerate}
\end{Remark}

\section{The complexified action}
\label{sec:C-star-n}

Let the torus $(S^1)^n$ act on a complex manifold $M$ by biholomorphisms.
If the manifold $M$ is compact, then the $(S^1)^n$-action
extends to a $(\C^*)^n$-action that is holomorphic 
not only in the sense that each element of $(\C^*)^n$ acts
by a biholomorphism but also in the sense that the action map 
$(\C^*)^n \times M \to M$ is holomorphic.
See, e.g., \cite[Theorem 4.4]{guillemin-sternberg}.
For the convenience of the reader, we briefly recall here
some of the details of this standard construction.

Let $\xi _1,\dots ,\xi _n$ be the fundamental vector fields of the
$(S^1)^n$-action with respect to the coordinate one-dimensional subtori. 
Let $J \colon TM \to TM$ be the multiplication by $\sqrt{-1}$.
We claim that the vector fields $-J\xi_1,\ldots,-J\xi_n$
are holomorphic 
(in the sense that their flows preserve the complex structure)
and commute with each other 
and with the vector fields $\xi_i$.

Because the $(S^1)^n$-action preserves $J$ and $\xi_j$,
it preserves $-J\xi_j$, for each $j$.
So the vector fields $-J\xi_j$ commute with the vector fields $\xi_i$
that generate this action.
Because $J$ is a complex structure, its Nijenhuis tensor,
$N(Z,W) := 2 \left( [JZ,JW] - J[Z,JW] - J[JZ,W] - [Z,W] \right)$, vanishes.  
Setting $Z= \xi_i$ and $W = \xi_j$, we get that
$[J\xi_i,J\xi_j] = J[\xi_i,J\xi_j] + J[J\xi_i,\xi_j] + [\xi_i,\xi_j]$,
and each of the three terms on the right hand side is zero.
So the vector fields $-J\xi_j$ commute with each other. 
A vector field $Y$ is holomorphic 
if and only if $[Y,JW] = J[Y,W]$ for each vector $W$; see 
\cite[Proposition 2.10 in Chapter IX]{kobayashi-nomizu}.
Set $Y := -J\xi_i$ and $W$ arbitrary;
because $JY(=\xi_i)$ is holomorphic, $[JY,JW] = J[JY,W]$;
by the vanishing of the Nijenhuis tensor,
\begin{equation*}
\begin{split}
 0 = N(JY,W) &= 2 \left( [-Y,JW] - J[JY,JW] - J[-Y,W] - [JY,W] \right) \\
			&=2([-Y,JW]-J[-Y,W]),
\end{split}
\end{equation*}
so $Y$ is holomorphic.

If $M$ is compact, 
the vector fields $-J\xi _1,\dots ,-J\xi _n$ are complete,
and we get an $\R^{2n}$-action, $\R^{2n} \times M \to M$, via
\begin{equation*}
 \left( \sum _{i =1}^{2n} a_i\e _i , x\right) \mapsto c_x(1),	
\end{equation*}
where $c_x(r)$ is the integral curve of the vector field 
$\sum_{i=1}^n -a_i J \xi_i + a_{n+i} \xi_i$ with $c_x(0)=x$.
This action descends to a $(\C^*)^n$-action by biholomorphisms that extends
the given $(S^1)^n$-action.  
Finally, the action map $(\C^*)^n \times M \to M$ is holomorphic,
because its differential, which at the point $(z,m)$ 
is the map $\C^n \times T_mM \to T_{z \cdot m} M$
that takes $(2\pi (r_1+i\theta_1 , \ldots, r_n+i\theta_n) , v)$ to 
$\sum_j -r_j J\xi_j|_{z \cdot m} + \theta_j \xi_j|_{z \cdot m} + z_* v$,
is complex linear.

\begin{Remark}
In the next section we will see that if there exists a fixed point
then the extended $(\C^*)^n$-action is faithful.
In general,
the extended $(\C^*)^n$-action might not be faithful.
\end{Remark}

\begin{Example} \label{complex-Cn} 
Let $(S^1)^n$ act on $\C^n$ with weights $\alpha_1,\ldots,\alpha_n$:
\begin{equation*}
	g\cdot (z_1,\dots ,z_n) = 
  (g^{\alpha _1}z_1,\dots , g^{\alpha _n}z_n),
\end{equation*} 
where
 $g^{\alpha _i}=g_1^{\alpha _{i1}}\dots g_n^{\alpha _{in}}$ 
for $g=(g_1,\dots ,g_n) \in (S^1)^n$ 
and for the isotropy weight 
$\alpha _i=(\alpha _{i1},\dots ,\alpha _{in}) \in \Z^n$.
Then the complexified action is given by the same formula
applied to $g=(g_1,\dots ,g_n) \in (\C^*)^n$.
\end{Example}

\section{Structures near fixed points}
\label{sec:Vp}

Let $M$ be a connected complex manifold of complex dimension $n$.
Let the torus $(S^1)^n$ act on $M$ faithfully by biholomorphisms.
Let $p$ be a point in $M$ that is fixed by the $(S^1)^n$-action.
Let $\alpha_1,\ldots,\alpha_n$ be the isotropy weights at $p$.

Let $\C_{\alpha_i}$ denote the one dimensional complex vector space $\C$
with the $(S^1)^n$-action that is obtained by composing the homomorphism
$(S^1)^n \to S^1$ that is encoded by the weight $\alpha_i$ with the standard
action of $S^1$ on $\C$ by scalar multiplication.

We begin with a local result:

\begin{lemm}\label{lemm:local}
There exists an $(S^1)^n$-invariant neighbourhood $U_p$ of $p$ in $M$,
an $(S^1)^n$-invariant neighbourhood $\widetilde{U}_p$
of the origin in $T_pM$,
and an $(S^1)^n$-equivariant biholomorphism
$\varphi_p \colon U_p \to \widetilde{U}_p$
whose differential at $p$ is the identity map on $T_pM$.
\end{lemm}

\begin{proof}
Let $\varphi \colon U \to \widetilde{U} \subseteq \C^n$ 
be a local holomorphic chart near $p$ with $\varphi (p)=0$. 
Identifying $\C^n$ with $T_pM$ via the differential
\begin{equation*}
  (d\varphi)_{p} \colon T_pM \to T_0\C^n \cong \C ^n,
\end{equation*}
we get a biholomorphism
\begin{equation*}
  \varphi' \colon U   \to  \widetilde{U}' \subseteq T_pM
\end{equation*}
whose differential at $p$ is the identity map on $T_pM$.
We want to obtain such a biholomorphism that is also equivariant.

Set
\begin{equation*}
   U' := \bigcap_{g \in (S^1)^n} gU.
\end{equation*}
Clearly, $U'$ is invariant and contains $p$.  We now show that $U'$ is open.
The complement of $U'$ is the image of the closed subset
$(S^1)^n \times (M \ssminus U)$ of $(S^1)^n \times M$ 
under the action map $(S^1)^n \times M \to M$.
Because $(S^1)^n$ is compact, the action map is proper.
Being proper means that the preimage of every compact set is compact;
when the target space $M$ is a manifold\footnote{
 In fact, it is enough to assume that the target space 
 is Hausdorff and compactly generated.
 Compactly generated means that a subset is closed if and only if
 its intersection with every compact set $K$ is closed in $K$;
 this property holds if the space is locally compact
 or if the space is metrizable.
}
it implies that the map is closed.
Thus, the complement $M \ssminus U'$ is closed, and so $U'$ is open. 

To obtain an equivariant chart, we average $\varphi'$: let
\begin{equation*}
 \widetilde {\varphi} := 
     \int_{g \in (S^1)^n} 
     (g \circ \varphi' \circ g^{-1}) \, dg \, \colon \, U' \, \to \, T_pM,
\end{equation*}
where $dg$ is Haar measure on $(S^1)^n$.
The map $\widetilde {\varphi}$ is holomorphic and $(S^1)^n$-equivariant.
Moreover, its differential at $p$ is the identity map on $T_pM$.
By the implicit function theorem, $\widetilde{\varphi}$
restricts to a biholomorphism from some smaller open neighbourhood $U''$ 
of~$p$ in~$M$ to an open neighbourhood of the origin in~$T_pM$.
The restriction of~$\widetilde{\varphi}$ to the invariant neighbourhood
$U_p := \bigcap\limits_{g \in (S^1)^n} g \cdot U''$ of~$p$ in~$M$
satisfies the requirements of the lemma.
\end{proof}

\begin{Corollary} \label{cor:polydisc}
There exists an $(S^1)^n$-equivariant local holomorphic chart
\begin{equation*} 
 \varphi_p \colon U_p \to \D^n 
\end{equation*}
from an invariant open neighbourhood $U_p$ of $p$
to a polydisc $\D^n$ in $\C_{\alpha_1} \oplus \ldots \oplus \C_{\alpha_n}$.
\end{Corollary}

\begin{proof}
By the definition of the isotropy weights,
there exists a complex linear $(S^1)^n$-equivariant isomorphism
between the tangent space $T_pM$ and the representation
$\C_{\alpha_1} \oplus \ldots \oplus \C_{\alpha_n}$.
Corollary~\ref{cor:polydisc} then follows from Lemma~\ref{lemm:local}
by restricting the chart to the preimage of a polydisc.
\end{proof}

We would like to extend the chart of Corollary \ref{cor:polydisc}
to a chart whose image is all of $\C^n$.
We can do this when the $(S^1)^n$ extends to a $(\C^*)^n$-action;
for example, if the manifold is compact;
by ``sweeping" by the $(\C^*)^n$-action.

\begin{Lemma} \label{lemm:Vp}
Suppose that the $(S^1)^n$-action extends to a $(\C^*)^n$-action.
Then there exists an invariant open neighbourhood $V_p$ of $p$ in $M$
and an $(S^1)^n$-equivariant biholomorphism of $V_p$
with $\C_{\alpha_1} \oplus \ldots \oplus \C_{\alpha_n}$. 
\end{Lemma}

\begin{proof}
Let $\varphi _p \colon U_p \to \D ^n$
be an $(S^1)^n$-equivariant holomorphic local chart,
as in Corollary~\ref{cor:polydisc}. 
Because $\varphi_p$ is $(S^1)^n$-equivariant and holomorphic,
it intertwines the restriction to $U_p$ of the vector fields
that generate the complexified $(\C^*)^n$-action on $M$
with the restriction to $\D^n$ of the vector fields that generate
the complexified $(\C^*)^n$-action on $\C^n=\C _{\alpha _1}\oplus \dots \oplus \C _{\alpha _n}$.
This, and the fact that $\varphi_p$ is a diffeomorphism between $U_p$
and $\D^n$, implies that $\varphi_p$ also intertwines the partial
flows on $U_p$ and on $\D^n$ that are generated by these vector fields;
in particular it intertwines the domains of definition of these partial flows.

For each $t \in \R$, let $g_t$ be the element of $(\C^*)^n$
that acts on $\C^n$ as scalar multiplication by $e^{-t}$,
and let $\eta \in \Lie(\C^*)^n$ be the generator of the one-parameter
subgroup $t \mapsto g_t$.
Because $e^{-t} \D^n \subset \D^n$ for all $t \geq 0$,
and because $\varphi_p$ intertwines the domains of definition
of the partial flows on $U_p$ and on $\D^n$ that correspond to $\eta$,
we get that $g_t U_p \subset U_p$ for all $t \geq 0$.
So, for every $t \geq 0$, the domain of definition of the 
$(S^1)^n$-equivariant biholomorphism
$$ \varphi_p^{(t)} := (g_t)^{-1} \circ \varphi_p \circ g_t
 \ \colon \ g_{-t} U_p \to e^t \D^n  $$
contains $U_p$.  
Here, $g_t \colon g_{-t} U_p \to U_p$
and $(g_t)^{-1} \colon \D^n \to e^t \D^n$
are given by the complexified actions on $M$ and on $\C^n$.
By the choice of $g_t$, the latter map is multiplication by $e^{t}$.

Moreover, because $\varphi_p$ intertwines the partial flows
that correspond to $\eta$
and these partial flows are defined for all $t \geq 0$,
the restriction to $U_p$ of $\varphi_p^{(t)}$
coincides with $\varphi_p$ for all $t \geq 0$.
Substituting $t-s$ instead of $t$, 
we get that the maps $\varphi_p^{(t)}$ and $\varphi_p^{(s)}$
agree whenever they are both defined.
Thus, all these maps fit together into a map
$$ \bigcup\limits_{t \geq 0} \varphi_p^{(t)}
 \colon V_p \to \C_{\alpha_1} \oplus \ldots \oplus \C_{\alpha_n}   ,$$
where $V_p = \bigcup_{t \geq 0} g_{-t} U_p$.
This map is onto, because its image is the union of the sets $e^t \D^n$ 
over all $t \geq 0$.
The map is one to one, because it is one to one on each $g_{-t} U_p$,
and for every two points in the domain there exists a $t \geq 0$ 
such that the points are both in $g_{-t} U_p$.
Because $V_p$ is covered by $(S^1)^n$-invariant open sets $g_{-t} U_p$
on which the map is an $(S^1)^n$-equivariant biholomorphism,
we deduce that the map is itself an $(S^1)^n$-equivariant biholomorphism,
as required. 
\end{proof}

\section{Obtaining a fan}
\label{sec:fan}

Let $M$ be a 
connected complex manifold of complex dimension $n$,
let the torus $(S^1)^n$ act on $M$ faithfully by biholomorphisms,
and assume that this action extends to a holomorphic $(\C^*)^n$-action.
The set of fixed points is discrete; assume that it is nonempty and finite. 

In Lemma~\ref{lemm:Vp} we assigned to every fixed point $p$ in $M$
an open subset $V_p$ that is biholomorphic to $\C^n$.  
By assumption, there exists at least one fixed point. So the union
$X$ of the sets $V_p$ over these fixed points,
\begin{equation*}
    X := \bigcup_{p \in M^{(S^1)^n}} V_p,
\end{equation*}
is nonempty. 

\begin{Remark} \label{connected}
In Section \ref{sec:compact} we show that if $M$ is compact and connected 
then the union $X$ of the sets $V_p$ is all of $M$.
The proof relies on the results of Sections~\ref{sec:fan}
and~\ref{sec:isomorphism}.
\end{Remark}

By its definition,
$X$ is a $(\C ^*)^n$-invariant open submanifold of $M$. 
Moreover, we claim that there exists a unique open $(\C^*)^n$ orbit
in $M$, this orbit and free and is dense in $M$, and it coincides with the
free $(\C^*)^n$ orbit in $V_p$ for each $p$.
To see this, we consider the fundamental vector fields 
$\xi_1,\dots ,\xi_n$ of the $(S^1)^n$-action with respect to the coordinate
one-dimensional subtori. We think of them as holomorphic sections 
$M \to T^{1,0}M \cong TM$ 
of the holomorphic tangent bundle $T^{1,0}M$ of $M$. 
The $n$-th exterior product $\bigwedge ^nT^{1,0}M \to M$ 
is a holomorphic line bundle 
and $\xi_1\wedge \dots \wedge \xi _n$ is a 
holomorphic section of this line bundle.
A point $x \in M$ belongs to an open $(\C ^*)^n$ orbit
if and only if 
$(\xi_1\wedge \dots \wedge \xi_n)(x)$ is not zero. 
This means that the union of 
the open $(\C ^*)^n$ orbits is the complement 
of the zero locus of a holomorphic section.
Because the zero locus is a complex analytic subvariety of $M$
and $M$ is connected, the union of the open $(\C^*)^n$ orbits
is either empty, or it is open, dense, and connected.
The claim then follows from the facts that 
there exists at least one $V_p$,
it contains a free and open $(\C^*)^n$ orbit,
and every two distinct orbits are disjoint.

In particular, $X$ is connected and dense in $M$. 

The connected components
of the fixed point sets of the circle subgroups of $(S^1)^n$
are closed complex submanifolds of $X$.  If such a submanifold
has complex codimension one, then, in analogy with the toric topology
literature, we call it a \emph{characteristic submanifold} of $X$
(cf.~\cite[p.~240]{masuda}).

Because $X$ is a union of finitely many $V_p$s
and each $V_p$ has only finitely many characteristic submanifolds,
there are only finitely many characteristic submanifolds in $X$.  
Denote them $$ X_1 , \ldots , X_m .$$

Let $T_i$ be the subgroup of $T$ that fixes $X_i$.
If a compact group acts faithfully on a connected manifold
then at every fixed point the linear isotropy representation is faithful.
Therefore, the linear isotropy representation of $T_i$
at any point $q$ of $X_i$ is faithful.   
Because $T_i$ acts holomorphically and fixes $X_i$, 
we get a faithful representation of $T_i$ on the one dimensional 
complex space $T_q X / T_q X_i$.
This gives an injection $T_i \to S^1$, where $S^1$ acts on $T_qX/T_qX_i$
by scalar multiplication.  
By continuity, this injection is independent of the choice of point $q$
in $X_i$.
Because, by assumption, $T_i$ contains a circle subgroup of $T$,
this injection is an isomorphism.  Let
$$ \lambda_i \colon S^1 \to T_i \subset (S^1)^n $$
be the inverse of this isomorphism, composed with the inclusion map 
into $(S^1)^n$.

We define an abstract simplicial complex:
\begin{equation*}
   \Sigma := \left\{ I \subseteq \{ 1,\dots ,m\} 
 \ \big| \ \bigcap_{i \in I} X_i \neq \emptyset \right \}.
\end{equation*}
To each simplex $I \in \Sigma$ we assign the cone 
\begin{equation*}
   C_I := \pos (\lambda _i \, \mid \, i \in I) := 
   \left\{ \sum_{i \in I} a_i\lambda _i \ \big| \ a_i \geq 0\right\}
\end{equation*}
in $\Lie (S^1)^n$. 

\begin{Example} \label{Cn}
Take $\C^n$ with coordinates $z_1, \ldots, z_n$.
Let $(S^1)^n$ act on it with weights $\alpha_1,\ldots,\alpha_n
 \in \Hom((S^1)^n,S^1) \subset (\Lie(S^1)^n)^*$.  
Suppose that the action is faithful; then $\alpha_1,\ldots,\alpha_n$
are a $\Z$-basis of $\Hom((S^1)^n,S^1)$.  
The characteristic submanifolds are the coordinate hyperplanes
$ \{ z_i = 0 \} $ for $i=1,\ldots,n$.
The homomorphisms $\lambda_1, \ldots, \lambda_n$ 
are the basis to $\Hom(S^1,(S^1)^n) \subset \Lie(S^1)^n$
that is dual to $\alpha_1,\ldots,\alpha_n$.
\end{Example}

Recall that
a cone in $\Lie (S^1)^n$ is \emph{unimodular}
if it is generated by part of a $\Z$-basis of $\Hom (S^1,(S^1)^n)$.

Returning to our general case -- 

\begin{Lemma}
The cones $C_I$, for $I \in \Sigma$, are unimodular.
\end{Lemma}

\begin{proof}
Let $I \in \Sigma$.  By the definition of $\Sigma$, this means
that the intersection $\bigcap_{i \in I} X_i$ is nonempty.
Let $q$ be a point in this intersection.  Let $p$ be a fixed point
such that $q \in V_p$.  Because $V_p$ is isomorphic 
to some $\C_{\alpha_1} \oplus \ldots \oplus \C_{\alpha_n}$
on which the action is faithful,
the lemma follows from Example~\ref{Cn}.
\end{proof}

Fix a point $q$ in the free $(\C^*)^n$ orbit in $X$.
For any $\xi \in \Lie (S^1)^n$, consider the curve 
$$ c_q^\xi \colon \R \to X $$
that is given by
$$ c_q^\xi(r) := \exp(-rJ\xi) \cdot q \qquad \text{ for } r \in \R $$
where $\exp \colon \Lie(\C^*)^n \to (\C^*)^n$ is the exponential map
and where $J$ denotes multiplication by $i$ in $\Lie(\C^*)^n$.

Denote by $C_I^0$ the relative interior of the cone $C_I$.
Denote 
$$ X_I = \bigcap _{i\in I}X_i \quad \text{and} \quad 
   X_I^0 = \bigcap_{i \in I} X_i \ssminus \bigcup_{j \not\in I} X_j. $$

\begin{Lemma} \label{lemm:limitX}
Let $\xi \in \Lie(S^1)^n$ and $I \in \Sigma$.
Then $\xi \in C_I^0$ if and only if the curve $c_q^\xi(r)$
converges as $r \to -\infty$ to a point $q'$ in $X_I^0$.
Moreover, in this case the limit point $q'$
belongs to $V_p$ for every $p$ such that $V_p \cap X_I \neq \emptyset$.
\end{Lemma}

\begin{proof}
Suppose that $\xi \in C_I^0$.
By the definition of $\Sigma$, \ $X_I$ is nonempty.  
Let $p$ be such that $V_p$ meets $X_I$.
Without loss of generality assume that $I = \{ 1, \ldots, k \}$
and that the characteristic submanifolds that meet $V_p$ are $X_1,\ldots,X_n$.
Let $\alpha_1,\ldots,\alpha_n$ denote the isotropy weights at $p$.
The assumption that $\xi \in C_I^0$ 
exactly means that $\langle \xi , \alpha_i \rangle$
is positive for $i=1,\ldots,k$ and zero for $i=k+1,\ldots,n$. 
Fix an isomorphism
$(z_1,\ldots,z_n) \colon V_p \to \C^n 
 = \C_{\alpha_1} \oplus \ldots \oplus \C_{\alpha_n}$
such that $z_i(q) = 1$ for all $i$.
In these coordinates, the curve $c_q^\xi(r)$ is represented as
\begin{equation*}
(z_1,\dots ,z_n)(c_q(r)) = 
 (e^{2 \pi r\langle \xi, \alpha_1\rangle}, \dots, 
  e^{2 \pi r\langle \xi, \alpha_n\rangle}).
\end{equation*}
As $r$ approaches $-\infty$, the curve in $\C^n$ approaches
the point $(\underbrace{0,\ldots,0}_k,\underbrace{1,\ldots,1}_{n-k})$.
On the other hand, the coordinates take each intersection $V_p \cap X_i$
to the coordinate hyperplane 
$\{ (z_1,\ldots,z_n) \ | \ z_i = 0 \}$,
and they take the intersection $V_p \cap X_I^0$
to the set
$\{ (z_1,\ldots,z_n) \ | \ z_i = 0 \text{ iff } 1 \leq i \leq k \}$.
So the curve approaches a point in $V_p \cap X_I^0$, as required.

Now suppose that the curve $c_q^\xi(r)$
converges as $r \to -\infty$ to a point in $X_I^0$.
Let $p$ be such that this limit point is contained in $V_p$.
As before, without loss of generality assume that $I = \{ 1, \ldots, k \}$
and that the characteristic submanifolds that meet $V_p$ 
are exactly $X_1, \ldots, X_n$;
fix an isomorphism
$(z_1,\ldots,z_n) \colon V_p \to \C^n 
   = \C_{\alpha_1} \oplus \ldots \oplus \C_{\alpha_n}$
such that $z_i(q) = 1$ for all $i$; the curve $c_q^\xi(r)$ is represented as
$ (z_1,\dots ,z_n)(c_q(r)) = 
  (e^{2 \pi r\langle \xi, \alpha_1\rangle}, \dots, 
   e^{2 \pi r\langle \xi, \alpha_n\rangle})$.
Because the curve approaches a limit as $r \to -\infty$,
the pairings $\langle \xi , \alpha_i \rangle$ are nonnegative
for all $i = 1 , \ldots , n$.
Because this limit is in $X_I^0$, the pairings 
are positive for every $i \in I$ and they vanish 
for every $i \in \{ 1, \ldots, n\} \ssminus I$.
Thus, $\xi \in C_I^0$ as required. 
\end{proof}

\begin{Corollary}
\begin{enumerate}
\item
For every $I, J \in \Sigma$, if $I \neq J$, 
then $C_I^0 \cap C_J^0 = \emptyset$.
\item
For every $I,J \in \Sigma$, 
$$C_I \cap C_J = C_{I \cap J}.$$
\item
The collection of cones 
\begin{equation*}
   \Delta: = \left\{ \, C_I \ \big| \ I \in \Sigma \, \right\}
\end{equation*}
is a fan, that is, every face of every cone in $\Delta$
is itself in $\Delta$, and the intersection of every two cones in $\Delta$
is a common face. 
\end{enumerate}
\end{Corollary}

\begin{proof}
Part (1) follows from Lemma~\ref{lemm:limitX}
because the sets $X_I^0$ are disjoint.
Part (3) follows from Part (2).

For Part (2), we only need to show the inclusion
$C_I \cap C_J \subseteq C_{I \cap J}$,
because the opposite inclusion is trivial.
Let $\xi \in C_I \cap C_J$.  Let $I' \subset I$ and $J' \subset J$
be the subsets such that $\xi \in C_{I'}^0$ and $\xi \in C_{J'}^0$.
Then $C_{I'}^0 \cap C_{J'}^0 \neq \emptyset$.
By Part (1), $I' = J'$.  Let $L = I' = J'$.
Then $L \subset I \cap J$, and $\xi \in C_{L}^0 \subset C_{I \cap J}$.
\end{proof}

\begin{Lemma} \label{XI-connected}
For every $I \in \Sigma$, the set $X_I$ is an
$(S^1)^n$-invariant smooth closed complex submanifold of $X$
of complex codimension $|I|$,
it is connected, and it contains a fixed point.
\end{Lemma}

\begin{proof}
Fix $I \in \Sigma$.  

Because each of the sets $X_i$, for $i \in I$, is closed in $X$,
so is the intersection $X_I$ of these sets.

Because $X$ is the union of open subsets $V_p$,
and because every intersection $V_p \cap X_I$
is an $(S^1)^n$-invariant complex submanifold of codimension $|I|$ in $V_p$,
we deduce that $X_I$ is itself an $(S^1)^n$-invariant complex submanifold
of codimension $|I|$ in $X$.
It remains to show that $X_I$ is connected and contains a fixed point. 

Choose any $\xi \in C_I^0$
(for example, we may take $\xi = \sum_{i \in I}\lambda_i$),
and choose any $q$ in the free $(\C^*)^n$ orbit in $X$.
By Lemma~\ref{lemm:limitX}, the curve $c_q^\xi(r)$ converges
as $r \to -\infty$; let $q'$ be its limit. 
Also by Lemma~\ref{lemm:limitX}, for every $p$
such that $V_p \cap X_I \neq \emptyset$, the limit point $q'$ belongs to $V_p$. 
Because $X_I$ is the union over such $p$ of the subsets $V_p \cap X_I$,
and because each of these subsets is connected and contains $q'$,
the union $X_I$ is connected.
Also, every $p$ such that $V_p \cap X_I \neq \emptyset$
belongs to $V_p \cap X_I$; because the set of such $p$s is nonempty,
$X_I$ contains a fixed point.
\end{proof}

\begin{Corollary} \label{cor:pure}
In the fan $\Delta$, every cone is contained in an $n$ dimensional cone.
\end{Corollary}

\begin{proof}
Every cone in the fan has the form $C_I$ for some $I \in \Sigma$.
By Lemma \ref{XI-connected}, the set $X_I$ contains a fixed point;
let $p$ be such a fixed point.  
Since $V_p$ was chosen as in Lemma~\ref{lemm:Vp}, 
by Example~\ref{Cn} there exist exactly $n$ characteristic submanifolds,
say, $X_j$ for $j \in J \subset \{ 1, \ldots, m \}$ with $|J|=n$, 
that pass through $p$.  
Then $J \in \Sigma$, 
and $C_J$ is an $n$ dimensional cone in $\Delta$ that contains $C_I$.
\end{proof}

\section{Isomorphism of the subset $X$ with a toric manifold}
\label{sec:isomorphism}

Let $M$ be a 
connected complex manifold of complex dimension $n$,
let the torus $(S^1)^n$ act on $M$ faithfully by biholomorphisms,
and assume that this action extends to a holomorphic $(\C^*)^n$-action.
The set of fixed points is discrete; assume that it is nonempty and finite.

In Section~\ref{sec:fan}
we described an open subset $X$ of $M$ and a unimodular fan $\Delta$.
Let $M_\Delta$ be the toric variety that is associated to the fan $\Delta$.

\begin{Lemma} \label{lemm:MDelta-X}
There exists an $(S^1)^n$-equivariant biholomorphism 
between $M_\Delta$ and $X$.
\end{Lemma}

We recall some properties of the set $X$ and the fan $\Delta$.
Let $F = M^{(S^1)^n}$ denote the fixed point set.
For every fixed point $p \in F$, let $\alpha_{p,1}, \ldots, \alpha_{p,n}$
denote the isotropy weights of the torus action at $p$.
\begin{enumerate}
\item
The set $X$ is the union over $p \in F$ of subsets $V_p$,
such that each $V_p$ is an invariant open neighbourhood 
of $p$ that is equivariantly biholomorphic to the linear representation
$\C_{\alpha_{p,1}} \oplus \ldots \oplus \C_{\alpha_{p,n}}$.
\item
The $n$-dimensional cones in $\Delta$ are in bijection 
with the fixed point sets $p \in F$, and the cone corresponding
to the fixed point $p$ is 
$\pos(\lambda_{i_1},\ldots,\lambda_{i_n})$, where 
$\lambda_{i_1},\ldots,\lambda_{i_n}$ is a basis of $\Lie(S^1)^n$
that is dual to the basis $\alpha_{p,1}, \ldots, \alpha_{p,n}$
of $(\Lie(S^1)^n)^*$.
\end{enumerate}

The toric variety $M_\Delta$ that is associated to the fan $\Delta$
has similar properties: it is the union over $p \in F$
of invariant subsets $V_p'$,
and every $V_p'$ is equivariantly biholomorphic 
to $\C_{\alpha_{p,1}} \oplus \ldots \oplus \C_{\alpha_{p,n}}$.

Lemma~\ref{lemm:MDelta-X} follows immediately 
from these properties of $X$ and $M_\Delta$, by the following lemma.

\begin{Lemma} \label{lemm:isomorphism}
Let $X$ and $X'$ be complex manifolds of complex dimension $n$,
equipped with holomorphic $(\C^*)^n$-actions.
Suppose that there exist open dense $(\C^*)^n$ orbits
$\calO$ in $X$ and $\calO'$ in $X'$.
Suppose that there exist invariant open subsets $V_p$ in $X$ and $V_p'$ in $X'$,
both indexed by $p \in F$, such that $X$ is the union of the sets $V_p$
and $X'$ is the union of the sets $V_p'$,
and that for each $p \in F$ there exists
an equivariant biholomorphism $\varphi_p \colon V_p \to V_p'$.
Then $X$ is equivariantly biholomorphic to $X'$.
\end{Lemma}

\begin{proof}
Necessarily, $\calO$ is contained in each $V_p$
and $\calO'$ is contained in each $V_p'$.
Fix a point $q$ in $\calO$ and a point $q'$ in $\calO'$.
After possibly composing each $\varphi_p$ by the action
of an element of $(\C^*)^n$, we may assume that $\varphi_p(q) = q'$
for each $p \in F$.  So, for each $p$ and $\tilde{p} \in F$,
the maps $\varphi_p$ and $\varphi_{\tilde{p}}$ coincide at the point $q$.
By equivariance, $\varphi_p$ and $\varphi_{\tilde{p}}$ coincide 
on all of $\calO$; by continuity, they coincide on the entire overlap
$V_p \cap V_{\tilde{p}}$.
Thus, the $\varphi_p$ fit together into a map
$$ \varphi = \bigcup_p \varphi_p \colon X \to X' .$$ 
This map is holomorphic, equivariant, and onto.
Similarly, the inverses $\psi_p := {\varphi_p}^{-1}$ fit together 
into a map
$$ \psi = \bigcup_p \psi_p \colon X' \to X .$$ 
We have that $\psi \circ \varphi = \text{id}_X$
and $\varphi \circ \psi = \text{id}_{X'}$;
thus, $\varphi \colon X \to X'$ is an equivariant biholomorphism,
as required.
\end{proof}

\section{The compact case}
\label{sec:compact}

Let $M$ be a connected complex manifold 
of complex dimension $n$,
with a faithful $(S^1)^n$-action, with fixed points. 

Suppose that $M$ is compact.  In Section~\ref{sec:C-star-n}
we extended the $(S^1)^n$-action to a holomorphic $(\C^*)^n$-action.
In Section~\ref{sec:fan} we described an open subset $X$ of $M$ 
and we associated to it a fan $\Delta$.

\begin{Lemma} \label{lemm:complete}
The fan $\Delta$ is complete.
\end{Lemma}

We begin by proving a special case:

\begin{Lemma} \label{lemm:complete0}
Let $M'$ be a complex manifold of complex dimension one,
equipped with a faithful holomorphic action of $S^1$ 
with at least one fixed point.
Suppose that $M'$ is compact and connected.
Then $M'$ is equivariantly biholomorphic to $\CP^1$
with a standard $\C^*$-action.
\end{Lemma}

\begin{proof}
Consider the $S^1$-action on $M'$.  Near a fixed point,
it is isomorphic to the restriction of either the standard $S^1$-action on $\C$
or the opposite $S^1$-action on $\C$
to an invariant neighbourhood of the origin in $\C$. 

Consider the flow that is generated by $-J\xi$, where $\xi$
generates the $S^1$-action.
If the $S^1$-action near a fixed point is standard,
then the trajectories of this flow converge to the fixed point
as their parameter approaches $-\infty$.
If the $S^1$-action near a fixed point is opposite from standard,
then the trajectories of this flow converge to the fixed point
as their parameter approaches $\infty$.

Outside the fixed point set, the action is free.
The quotient $M'/S^1$ is\footnote
{
Here, ``is" means
that there exists a unique manifold-with-boundary structure on $M'/S^1$
such that a function is smooth if and only if its pullback to $M'$ is smooth.
}
a real one-manifold with boundary;
its boundary is exactly the image of the fixed point set. 
Because $M'$ is compact and connected and contains a fixed point, and by the
classification of one-manifolds, the quotient $M'/S^1$
must be a closed segment.

The flow on $M'$ that is generated by $-J\xi$ descends to a flow
on the interior of $M'/S^1$ that does not have fixed points.
For each boundary component, the flow approaches that component
either as its parameter approaches $\infty$ or as the parameter
approaches $-\infty$.  Necessarily, it approaches one boundary
component when the parameter approaches $\infty$
and it approaches the other boundary component
when the parameter approaches $-\infty$.

The corresponding fan must then be equal to the fan of $\CP^1$,
and the manifold is equivariantly biholomorphic to $\CP^1$
by Lemma~\ref{lemm:isomorphism}.
\end{proof}

We now return to the setup of Lemma~\ref{lemm:complete}:
We have a connected complex manifold $M$ 
of complex dimension $n$,
with a faithful $(S^1)^n$-action, with fixed points. 
We assume that $M$ is compact.
We consider the open subset $X$ of $M$ 
and the associated fan $\Delta$
as described in Section~\ref{sec:fan}.

\begin{Lemma} \label{lemm:no-facets}
Every $n-1$ dimensional cone in $\Delta$ is a common face
of two $n$ dimensional cones in $\Delta$.
\end{Lemma}

\begin{proof}
Let $C_I$ be an $n-1$ dimensional cone in $\Delta$,
corresponding to the subset
$I = \{ i_1,\ldots,i_{n-1} \}$ of $\{ 1, \ldots, m \}$.

Let $T_I$ be the codimension one subtorus of $(S^1)^n$
that is generated by the circles $T_i$ for $i \in I$.
By Lemma~\ref{XI-connected}, $X_I$ is a connected complex manifold
of dimension one, equipped with a faithful holomorphic action
of the circle $(S^1)^n/T_I$ with at least one fixed point.
We will now show that $X_I$ is compact,
and will deduce Lemma~\ref{lemm:no-facets} 
from Lemma~\ref{lemm:complete0}.

First note that $X_I$ is a connected component 
of the fixed point set of $T_I$ in $X$.
This follows from the facts that $X_I$ is connected 
(by Lemma~\ref{XI-connected})
and that, 
for each of the subsets $V_p$,
if the intersection $V_p \cap X_I$
is nonempty then it is a connected component of the fixed point set 
of $T_I$ in $V_p$.
Let $N$ denote the connected component of the fixed point set
of $T_I$ in $M$ that contains $X_I$.
As in any holomorphic torus action on a complex manifold,
$N$ is an $(S^1)^n$-invariant closed complex submanifold of $M$.
By examining $N$ near a point of $X_I$, we deduce that 
$N$ has complex dimension one.
Because $N$ is closed in $M$ and $M$ is compact, $N$ is compact.
By Lemma~\ref{lemm:complete0}, $N$ is equivariantly biholomorphic
to $\CP^1$ with a standard action of the circle $(S^1)^n/T_I$.
In particular, $N$ contains two fixed points; denote them $p'$ and $p''$. 
The intersection $V_{p'} \cap N$, being a $(\C^*)^n$-invariant
neighbourhood of $p'$ in $N$, must be all of $N \ssminus \{ p'' \}$.
Similarly, 
the intersection $V_{p''} \cap N$ is all of $N \ssminus \{ p' \}$.
So $N$ is contained in the union $X$ of the sets $V_p$,
and so $N$ must be equal to $X_I$.
Thus, $X_I$ is equivariantly biholomorphic to $\CP^1$ 
with a standard action of the circle $(S^1)^n/T_I$.
This implies the result of Lemma~\ref{lemm:no-facets}.

\end{proof}

We are now ready to prove Lemma~\ref{lemm:complete}.

\begin{proof}[Proof of Lemma~\ref{lemm:complete}]
Let $|\Delta|$ denote the union of the cones in $\Delta$,
and let $|\Delta^{n-2}|$ denote the union of the cones in $\Delta$
that have codimension $\geq 2$.
The complement $\Lie(S^1)^n \ssminus |\Delta^{n-2}|$
is connected, open, and dense in $\Lie(S^1)^n$.

By Lemma~\ref{lemm:no-facets}, the union of the relative
interiors of the faces of $\Delta$ of dimension $(n-1)$ 
and of dimension $n$ is open in $\Lie(S^1)^n$.
This union is $|\Delta| \ssminus |\Delta^{n-2}|$.
Thus, $|\Delta| \ssminus |\Delta^{n-2}|$ is also open in 
$\Lie(S^1)^n \ssminus |\Delta^{n-2}|$.

But because $|\Delta|$ is closed in $\Lie(S^1)^n$,
we also have that $|\Delta| \ssminus |\Delta^{n-2}|$ is closed
in $\Lie(S^1)^n \ssminus |\Delta^{n-2}|$.

Because $|\Delta| \ssminus |\Delta^{n-2}|$
is open and closed in $\Lie(S^1)^n \ssminus |\Delta^{n-2}|$
and $\Lie(S^1)^n \ssminus |\Delta^{n-2}|$ is connected,
we deduce that $|\Delta| \ssminus |\Delta^{n-2}|$
is either empty or is equal to all of 
$\Lie(S^1)^n \ssminus |\Delta^{n-2}|$.

Because, by assumption, $M$ has a fixed point, 
$\Delta$ has at least one $n$ dimensional cone,
so $|\Delta| \ssminus |\Delta^{n-2}|$ is not empty.
So $|\Delta| \ssminus |\Delta^{n-2}|$ 
is equal to all of $\Lie(S^1)^n \ssminus |\Delta^{n-2}|$.
Taking the closures,
we deduce that $|\Delta| = \Lie(S^1)^n$, as required.
\end{proof}

We are now ready to prove our main theorem.

\begin{proof}[Proof of Theorem~\ref{main-theorem}.]
Lemma~\ref{lemm:isomorphism} gives an equivariant biholomorphism
$$ \varphi \colon M_\Delta \to X.$$
By Lemma~\ref{lemm:complete}, the fan $\Delta$ is complete.
This implies that the toric variety $M_\Delta$ is compact. 
So $X$ must be compact.
Because $M$ is Hausdorff and connected, 
and $X$ is a subset that is both compact and open, $X$ is all of $M$.
So $\varphi$ defines an equivariant biholomorphism
from $M_\Delta$ to $M$, as required.
\end{proof}

\subsection*{Acknowledgements}
We would like to thank the anonymous referee for helpful suggestions.
The second author would also like to thank Ignasi Mundet i Riera
and Jaume Amor\'os for helpful discussions.

\end{document}